\documentclass[reqno,a4paper,12pt]{amsart} 

\usepackage{amsmath,amscd,amsfonts,amssymb}
\usepackage{mathrsfs,dsfont}

\usepackage{bbm}

\allowdisplaybreaks

\numberwithin{equation}{section}
\numberwithin{figure}{section}

\renewcommand\le{\leqslant}
\renewcommand\ge{\geqslant}
\renewcommand\leq{\leqslant}
\renewcommand\geq{\geqslant}

\theoremstyle{plain}
\newtheorem{thm}{Theorem}[section]

\newtheorem{lemma}[thm]{Lemma}
\newtheorem{cor}[thm]{Corollary}

\newtheorem*{claim*}{Claim}
\newtheorem{theorem}{Theorem}

\theoremstyle{definition}

\newtheorem*{definition*}{Definition}
\newtheorem*{remarks*}{Remarks}
\newtheorem*{remark*}{Remark}

\begin{document}
	
	\title
	[Norm bounds on Fourier series with polynomial spectra.]{Norm bounds on Fourier series with polynomial spectra and constrained coefficients.}
	
	\author{Ioann Vasilyev}
	\address{
		St. Petersburg Department of Steklov Mathematical Institute, Fontanka 27, St. Petersburg 191023, Russia}
	\email{ivasilyev@pdmi.ras.ru}

	\thanks{The author is supported by the Foundation for the Advancement of Theoretical Physics and Mathematics ``BASIS''. \\ This work was supported by the Ministry of Science and Higher Education of the Russian Federation (agreement 075-15-2025-344 dated 29/04/2025 for Saint Petersburg Leonhard Euler International Mathematical Institute at PDMI RAS)}

	\begin{abstract}
We prove two endpoint $L^4$ estimates for Fourier sums with
arithmetic spectra. First, let $P\in\mathbb Z[x]$ have degree
$d\geq3$. For every complex sequence $(a_k)$ with nonincreasing
moduli, we show that
\[
 \left\|\sum_{k=1}^N a_k e^{iP(k)\theta}\right\|_{L^4(\mathbb T)}
 \lesssim_P
 \left(\sum_{k=1}^N|a_k|^2\right)^{1/2}.
\]
Thus polynomial spectra of degree at least three satisfy an
endpoint coefficient-restricted $\Lambda(4)$ estimate. The proof
combines optimal additive-energy bounds for polynomial spectra with
Littlewood--Paley theory and Abel summation.

For the quadratic spectrum we establish the sharp weighted estimate
\[
 \left\|\sum_{k=1}^N d_k e^{ik^2\theta}\right\|^2_{L^4(\mathbb T)}
 \lesssim
\sum_{k=1}^N \,|d_k|^2,
\]
once complex $d_k$ satisfy $|d_k|\asymp k^{-1/2}$. In particular, this disproves
Conjecture~4.5 of Eceizabarrena and Da Rocha~\cite{EceiDaR24} and determines the sharp growth rate of the $L^4$ norm there. Our results also  strengthen and complement certain theorems by S. Bochkarev~\cite{Bochkarev14} and A. Córdoba~\cite{Cordoba24}.
\end{abstract}

	\maketitle

\section{Introduction}

Let \(\mathbb T=\mathbb R/(2\pi\mathbb Z)\) denote the unit circle which we identify with the interval \([0,2\pi)\), equipped with the normalized Lebesgue measure. We begin by recalling a theorem by A. C\'ordoba,~\cite{Cordoba24}. 

\begin{theorem}
\label{cordoba} 
Let 
\[ G(\theta)=\sum_{k=1}^N b_k e^{ik^2\theta},\] 
where \(\theta\in\mathbb T\), \(N\in\mathbb N\), \(b_k\in\mathbb R_+\), and let the sequence \((b_k)_{k=1}^N\) be decreasing. Then it holds that
\[ \|G\|_{L^q(\mathbb T)}\lesssim_q \|G\|_{L^2(\mathbb T)} 
\] 
for all \(2<q<4\).
\end{theorem} 

 Here and everywhere below notation $X\lesssim_\alpha Y$ means that $X\le C\cdot Y$ for some absolute constant $C > 0$ that might depend only on a parameter $\alpha$. We write $X\asymp_\alpha Y$ once we have  $X\lesssim_\alpha Y$ and $X\gtrsim_\alpha Y$ simultaneously.

The results considered in this paper belong to the circle of questions around \(\Lambda(p)\)-sets and Fourier series with sparse spectra. Recall that a set of integers \(\Lambda\) is called a \(\Lambda(p)\)-set if trigonometric polynomials with spectrum in \(\Lambda\) satisfy an \(L^p\)-estimate in terms of their \(L^2\)-norm. In this language, C\'ordoba's theorem may be viewed as a coefficient-restricted \(\Lambda(p)\)-type estimate for the square spectrum. For other results in this direction, but of different type, see papers~\cite{Gabdullin24} and~\cite{GabdullinKonyagin24}.

This result is sharp in the sense that, in general, one cannot take \(q=4\); see Rudin,~\cite{Rudin60}. The main thing we do in this paper is we prove the following analogue for polynomial phases of degree at least three. 
\begin{theorem}
\label{higherpowers} 
Let \(P\in\mathbb Z[x]\) be a nonconstant polynomial of degree \(d\ge3\). Let \[ F(\theta)=\sum_{k=1}^N a_k e^{iP(k)\theta}, \] 
where \(\theta\in\mathbb T\), \(N\in\mathbb N\), \(a_k\in\mathbb C\), and let the sequence \((|a_k|)_{k=1}^N\) be decreasing. Then it holds that \[ \|F\|_{L^4(\mathbb T)} \lesssim_P \|F\|_{L^2(\mathbb T)}. \]
\end{theorem} 

In other words, thanks to the Parseval theorem, we wish to prove the following Main Inequality
\begin{equation}
\label{main} 
\left\|\sum_{k=1}^N a_k e^{iP(k)\theta}\right\|_{L^4(\mathbb T)} \lesssim_P \left(\sum_{k=1}^N |a_k|^2\right)^{1/2}.
\tag{MI}
\end{equation}

Let us also mention the relation of our result with one result by Bochkarev, see~\cite{Bochkarev14}. For the pure power spectra \(k^d\), \(d\ge3\), Bochkarev obtained \(L^1\)-lower estimates for Fourier sums under a local regularity assumption on the moduli of the coefficients. In particular, his result applies to unimodular coefficients. Theorem~\ref{higherpowers} gives, for polynomial spectra of degree at least three and monotone moduli, an \(L^4\)-upper estimate, from which the corresponding \(L^1\)-lower estimate follows by interpolation. 

We also recall that another part of Bochkarev's work concerns Littlewood--Paley descriptions of the space \(\mathrm{BMO}\). This direction was developed further in \cite{TselishchevVasilyev20}, where more general Littlewood--Paley  characterizations of \(\mathrm{BMO}\) and Triebel--Lizorkin spaces were proved. Here we will only use the classical Littlewood--Paley inequality as one of our principal tools.



\bigskip

Here is our second main result of this paper.
\begin{theorem}
\label{bochkarev} 
Let 
\[ f(\theta)=\sum_{k=1}^N d_k e^{ik^2\theta},\] 
where \(\theta\in\mathbb T\), \(N\in\mathbb N\), \(d_k\in\mathbb C\), and suppose that \(|d_k|\asymp k^{-1/2}\). Then it holds that
\[ \|f\|_{L^4(\mathbb T)}\lesssim \|f\|_{L^2(\mathbb T)}.
\] 
\end{theorem} 

In particular, our Theorem~\ref{bochkarev} gives a complete answer to Conjecture 4.5 of Eceizabarrena and Da Rocha~\cite{EceiDaR24} and determines the sharp order of growth in the critical $L^4$ case.

As a direct corollary of Theorem~\ref{bochkarev}, we obtain the following inequality.
\begin{cor}
The following estimate holds for any arrangement of signs  
\[
\sqrt{\log N}\lesssim \int_0^{2\pi} \biggl|\sum_{n=1}^N\pm\frac{1}{\sqrt{n}}e^{in^2x}\biggr|dx.
\]
\end{cor} 

Let us also point out that our estimates are naturally related to the uncertainty principle in harmonic analysis. Indeed, they express the fact that a function whose Fourier transform is supported on a sufficiently arithmetically sparse set cannot have its mass concentrated too strongly in physical space, see for instance the book~\cite{HavinJoricke94}. In the present paper this phenomenon is quantified through \(L^4\)-bounds and the resulting \(L^1\)-lower estimates for polynomial spectra.

It is natural to ask whether our results remain valid when the monotonicity assumption on the moduli of coefficients is replaced
by an appropriate bounded variation or general monotonicity condition. Related conditions for general trigonometric series
have been studied by S. Tikhonov, see~\cite{Tikhonov07,Tikhonov21}. However, we prefer not to pursue those questions here.

\section{Proof of Theorem~\ref{higherpowers}}
Changing the constant in the inequality~\eqref{main} if necessary, we may assume that $P$ is strictly monotone. Replacing \(P\) by \(-P\), if needed, we may also assume that the leading coefficient of \(P\) is positive. Thus, for notational simplicity, after adding a constant to \(P\) if necessary we may assume that this polynomial is strictly increasing and positive on \(\mathbb N\).

 We first prove an auxiliary result on exponential sums. 
\begin{lemma}
\label{keylem} 
Let \(M\in \mathbb N\) and denote for \(\theta\in \mathbb T\)
\[ H_M(\theta)=\sum_{k=1}^M \varepsilon(k)e^{iP(k)\theta}, \] where \(\varepsilon:\{1,\ldots,M\}\to \{z\in \mathbb C: |z|=1\}\). Then \[ \|H_M\|_{L^4(\mathbb T)}\lesssim_P M^{1/2}. \] 
\end{lemma} 
\begin{proof} 
For an integer \(n\), denote
 \[ r_{P+P,M}(n) = \#\{(u,v)\in\{1,\ldots,M\}^2:\ P(u)+P(v)=n\}. 
\] 
Write \[ H_M(\theta)^2=\sum_{n\in \mathbb Z} c_n e^{in\theta}, \] where \[ c_n= \sum_{\substack{1\le u,v\le M\\ P(u)+P(v)=n}} \varepsilon(u)\varepsilon(v), \]
with the usual convention that \(c_n=0\) if the equation \(n=P(u)+P(v)\) has no solutions.

Parseval theorem tells us that 
\[ \|H_M\|_{L^4(\mathbb T)}^4 = \|H_M^2\|_{L^2(\mathbb T)}^2 = \sum_{n\in \mathbb Z} |c_n|^2. \] Since \(|c_n|\le r_{P+P,M}(n)\), we obtain \[ \|H_M\|_{L^4(\mathbb T)}^4 \le \sum_{n\in \mathbb Z} r_{P+P,M}(n)^2. \] 
By the results of  Browning~\cite{Browning15},~\cite{Browning24} in degree $d$ four and higher, 
\[ \sum_{n\in \mathbb Z} r_{P+P,M}(n)^2\lesssim_P M^2. \] 
For \(d=3\), the required estimate follows from Wooley
\cite{Wooley00}. Therefore \[ \|H_M\|_{L^4(\mathbb T)}^4\lesssim_P M^2, \] and the lemma follows. 
\end{proof} 

We now return to the proof of Theorem~\ref{higherpowers}. 

\begin{proof} 
 Write $a_k=\varrho_k\varepsilon(k), \varrho_k=|a_k|$, and $|\varepsilon(k)|=1,$  where the sequence \((\varrho_k)_{k=1}^N\) is decreasing. 

For every integer \(j\geq 0\), denote 
\[ I_j:=\{k\in\{1,\ldots,N\}:\ 2^j\le P(k)<2^{j+1}\}. \] 
Notice that whenever \(I_j\neq\emptyset\), it is an interval, so we may write as follows \[ I_j=\{u_j,u_j+1,\ldots,v_j\}, \] 
with certain natural numbers \(u_j\) and \(v_j\). Notice that since \(P(k)\asymp_P k^d\), we have, for all relevant \(j\), the following properties
$u_j\asymp_P 2^{j/d}, v_j\asymp_P 2^{j/d},$ and $u_j-u_{j-1}\asymp_P 2^{j/d}.$ Define the corresponding Littlewood--Paley blocks by the formula 
\[ \Delta_jF(\theta) = \sum_{k\in I_j} a_k e^{iP(k)\theta}, \]
if \(I_j\) is nonvoid and by \(\Delta_jF=0\) identically otherwise. Then we obviously have
\[ F=\sum_{j=0}^\infty \Delta_jF. \]

Consider an auxiliary exponential sum 
\[ S_m(\theta)=\sum_{k=1}^m \varepsilon(k)e^{iP(k)\theta}. \]
By the Abel summation formula, we get
 \[ \begin{split} \Delta_jF &= \sum_{k=u_j}^{v_j}\varrho_k(S_k-S_{k-1})\\ &= \varrho_{v_j}S_{v_j} -\varrho_{u_j}S_{u_j-1} + \sum_{k=u_j}^{v_j-1}(\varrho_k-\varrho_{k+1})S_k, \end{split} \] 
and therefore, Lemma~\ref{keylem} yields the following bound 
\[ 
\begin{split} 
\|\Delta_jF\|_{L^4(\mathbb T)} &\le \varrho_{v_j}\|S_{v_j}\|_{L^4(\mathbb T)} +\varrho_{u_j}\|S_{u_j-1}\|_{L^4(\mathbb T)} +\sum_{k=u_j}^{v_j-1}(\varrho_k-\varrho_{k+1})\|S_k\|_{L^4(\mathbb T)}\\ &\lesssim_P v_j^{\frac{1}{2}} \left( \varrho_{v_j} +\varrho_{u_j} +\sum_{k=u_j}^{v_j-1}(\varrho_k-\varrho_{k+1}) \right). 
\end{split} 
\] 
Notice that in the line just above we have used the fact $\varrho_k\geq 0$ and that the sequence \((\varrho_k)_{k=1}^N\) is decreasing.

By telescoping, we also infer the following property 
\[ \varrho_{v_j} +\varrho_{u_j} +\sum_{k=u_j}^{v_j-1}(\varrho_k-\varrho_{k+1}) = 2\varrho_{u_j}, \] 
from where we deduce that 
\[ \|\Delta_jF\|_{L^4(\mathbb T)} \lesssim_P v_j^{1/2}\varrho_{u_j} \lesssim_P 2^{j/(2d)}\varrho_{u_j}. \] 
Using first the Littlewood--Paley inequality on \(\mathbb T\) and then Minkowski inequality we also get the following estimates
\[ \|F\|_{L^4(\mathbb T)}^2 \lesssim  \left\|\biggl(\sum_{j=0}^\infty|\Delta_jF|^2\biggr)^{\frac{1}{2}}\right\|_{L^4(\mathbb T)}^2 \le \sum_{j=0}^\infty \|\Delta_jF\|_{L^4(\mathbb T)}^2. \] 
Consequently, 
\[ \|F\|_{L^4(\mathbb T)}^2 \lesssim_P \sum_{j=0}^\infty 2^{\frac{j}{d}}\varrho_{u_j}^2. \] 

It remains to compare the last expression with the \(L^2\)-norm of the function \(F\). Since the sequence \((\varrho_k)_{k=1}^N\) is decreasing and since \(u_j-u_{j-1}\asymp_P 2^{j/d}\), we infer the following estimates 
\[ \sum_{k=u_{j-1}}^{u_j-1}\varrho_k^2 \ge (u_j-u_{j-1})\varrho_{u_j}^2 \gtrsim_P 2^{\frac{j}{d}}\varrho_{u_j}^2. \] 
Summing in \(j\), we get the bound
\[ \sum_{j=1}^\infty 2^{\frac{j}{d}}\varrho_{u_j}^2 \lesssim_P \sum_{k=1}^N \varrho_k^2, \] 
which permits us to conclude that
\[ \|F\|_{L^4(\mathbb T)}^2 \lesssim_P \sum_{k=1}^N |a_k|^2. \] 
Since the frequencies \(P(k)\) are distinct we also know, by the Parseval theorem, that 
\[ \|F\|_{L^2(\mathbb T)}^2 = \sum_{k=1}^N |a_k|^2, \] 
which finishes off the proof of Theorem~\ref{higherpowers}. 
\end{proof}

 \section{Proof of Theorem~\ref{bochkarev}}

The case $N=1$ being trivial, we thus assume in what follows that $N\geq 2$.

Since $n^{-1/2}\lesssim |d_n|$, we know that 
$$
\|f\|_2^2=\sum_{n\leq N} |d_n|^2 \gtrsim \log N.
$$
On the other hand, denote
\[
W(m):=\sum_{\substack{m=a^2+b^2\\ a,b\geq 1}}
\frac{1}{\sqrt{ab}}.
\]
As in Theorem~\ref{higherpowers} we infer the property
$$
\|f\|_4^4=\sum_{m=1}^{2N^2} \biggl|\sum_{\substack{m=a^2+b^2\\ 1\leq a,b\leq N}} d_a d_b \biggr|^2 \lesssim \sum_{m=1}^{2N^2} W^2(m).
$$
Thus, we need to prove that
\[
\sum_{1\leq m\leq 2N^2}W^2(m)\lesssim (\log N)^2.
\]

We first decompose the range of $m$ in the sum just above into dyadic intervals satisfying $2^j\leq m<2^{j+1}$ with $j\in \mathbb N$. Hence, it is sufficient to prove that, for every $M:=2^j\leq 2N^2$, it holds that
\[
\Sigma_M:=\sum_{M\leq m<2M}W^2(m)\lesssim \log(2M).
\tag{$\diamond$}
\]
Indeed, if this is true, then
\[
\sum_{1\leq m\leq 2N^2}W^2(m)
\lesssim
\sum_{j=0}^{\lceil \log_2(2N^2)\rceil}\log(2^{j+1})
\lesssim (\log N)^2.
\]

So, we concentrate on the proof of inequality $(\diamond)$. To this end, for $U:=2^k$ and $V:=2^\ell$ with $k$ and $\ell$ natural, denote
\[
r_U(m):=\#\{(a,b):a^2+b^2=m,\ U\leq ab< 2U\},
\]
and put
\[
E_{U,V}(M):=
\sum_{M\leq m<2M}r_U(m)r_V(m).
\]
Since
\[
W(m)\lesssim
\sum_{k=0}^{\lceil\log_2(2M)\rceil}
\frac{r_{2^k}(m)}{2^{k/2}},
\]
we obtain
\[
\Sigma_M
\lesssim
\sum_{k,\ell=0}^{\lceil\log_2(2M)\rceil}
\frac{E_{2^k,2^\ell}(M)}{2^{(k+\ell)/2}}.
\]

Hence, it suffices to prove the following bound
\[
E_{U,V}(M)
\lesssim
{\bf 1}_{U=V}U+
\frac{UV\log(2M)}{M}.
\tag{1}
\]
Indeed, inequality (1) gives
\[
\begin{split}
\Sigma_M
&\lesssim
\sum_{k=0}^{\lceil\log_2(2M)\rceil}1
+
\frac{\log(2M)}{M}
\sum_{k,\ell=0}^{\lceil\log_2(2M)\rceil}
2^{(k+\ell)/2}
\\
&=
\sum_{k=0}^{\lceil\log_2(2M)\rceil}1
+
\frac{\log(2M)}{M}
\left(
\sum_{k=0}^{\lceil\log_2(2M)\rceil}2^{k/2}
\right)^2
\\
&\lesssim
\log(2M).
\end{split}
\]

So, we concentrate on $E_{U,V}(M)$, which is nothing but the
number of solutions $(a,b,c,d)$ of the equation
\[
a^2+b^2=c^2+d^2
\tag{eq}
\]
such that $a^2+b^2\in[M,2M), ab\in[U,2U),$ and $cd\in[V,2V)$.

\bigskip

The diagonal solutions are easy to handle. Indeed, after interchanging
$a$ and $b$ if necessary, we can suppose that $a\leq b$. Then 
$$ \frac{\sqrt{M}}{\sqrt{2}}\leq b<\sqrt{2M} \;\;\; \text{ and  } \;\;\; \frac{U}{\sqrt{2M}}<a\leq \frac{2\sqrt{2}U}{\sqrt{M}}.$$
Hence the number of possible pairs $(a,b)$ is bounded from above, up to a multiplicative constant by
\[
\frac{U}{\sqrt M}\sqrt M
\lesssim U.
\]
Thus, the diagonal contribution is bounded by ${\bf 1}_{U=V}U.$

\bigskip

We now consider the off-diagonal solutions. Let $E_{U,V}^{\mathrm{off}}(M)$ denote the number of such solutions. By symmetry, we can
suppose that $U\leq V, a\leq b,$ and $c\leq d.$
We consider the case $c>a$, so that necessarily $b>d$; the other
case is treated analogously. Put $h:=c-a$, and $q:=b-d.$
Then
\[
(c+a)h=(b+d)q.
\tag{2}
\]

Denote $A:=\frac{U}{\sqrt M}$ and $C:=\frac{V}{\sqrt M}$, so that $A\leq C$.
We thus have 
$$\frac{A}{\sqrt{2}}<a\leq 2\sqrt{2}A, \;\;\;  \frac{V}{\sqrt{2M}}<c\leq \frac{2\sqrt{2}V}{\sqrt{M}}, \;\;\; \text{ and } \;\;\; \frac{\sqrt{M}}{\sqrt{2}}\leq d<\sqrt{2M}.$$ 
We decompose the possible values of $h$ into the dyadic intervals defined by $2^s\leq h< 2^{s+1},$ for $s=0,1,\ldots .$
Fix such an $s$ and put $H:=2^s.$
It follows from (2) that
\[
\frac{HV}{4M}<q
\left(=
\frac{h(c+a)}{b+d}\right)
\leq \frac{8HV}{M}.
\]
Put
\[
Q:=\frac{HV}{M}.
\tag{3}
\]
Thus, $Q/4<q\leq 8Q.$
Since $V< M$, we also have
\[
Q< H.
\tag{4}
\]

Since $c=a+h,$
we have $c^2-a^2=h(h+2a),$
and therefore
\[
q\mid h(h+2a).
\tag{5}
\]
For $k\geq 0$, denote
\[
\rho_k(q):=
\#\{x \text{ modulo } q \text{ such that } x^2 \underset{q}{\equiv} k \}.
\]
After the change of variables
\[
x:=a+h,
\]
condition (5) becomes $x^2 \underset{q}{\equiv} a^2.$
Hence, for fixed $a$ and $q$, there are exactly $\rho_{a^2}(q)$
admissible residue classes for $h$ modulo $q$. Consequently, the
number of $h$ in an interval of length $H$ satisfying (5) is at most
\[
\left(\frac{H}{q}+1\right)\rho_{a^2}(q).
\]
Recall inequality (4), which tells us that $q<8H$, and hence
\[
\left(\frac{H}{q}+1\right)\rho_{a^2}(q)
<
\frac{9H}{q}\rho_{a^2}(q).
\tag{6}
\]

Moreover, a fixed triple $(a,h,q)$ determines the original solution
in at most one way. Indeed, $c=a+h$ while
$b-d=q$ and, by (2), we have $b+d=\frac{h(h+2a)}{q}.$
Therefore
\[
b=
\frac12\left(
\frac{h(h+2a)}q+q
\right),
\qquad
d=
\frac12\left(
\frac{h(h+2a)}q-q
\right).
\]

Let $S_H$ denote the number of off-diagonal solutions to equation (eq) corresponding
to $H\leq h< 2H.$ By properties (3)--(6),
\[
S_H
\lesssim
H
\sum_{\frac{A}{\sqrt{2}}<a\leq 2\sqrt{2}A}
\sum_{\frac{Q}{4}<q\leq 8Q}
\frac{\rho_{a^2}(q)}q.
\tag{7}
\]

We now estimate the inner sum. Write $a=2^t\Omega,$ where $\Omega$ is odd, and put
\[
P_a(X):=
\sum_{q\leq X}\rho_{a^2}(q).
\]
By Lemma 4 of the paper~\cite{Lapkova17} by  K. Lapkova (see also formula (3.11) therein) we know that
\[
P_a(X)
=
\sum_{d\mid\Omega}
\sum_{\substack{\alpha\geq0\\2^\alpha d^2\leq X}}
\kappa_\alpha d
\sum_{h\leq \frac{X}{2^\alpha d^2}}
\xi_d(h),
\tag{8}
\]
where
\[
\xi_d(n)
=
\sum_{\ell m=n}\mu^2(\ell)\chi_d(m),
\tag{9}
\]
and $\chi_d$ is the principal character modulo $2\Omega/d$, that is,
\[
\chi_d(m)=
\begin{cases}
1,&(m,2\Omega/d)=1,\\
0,&(m,2\Omega/d)>1.
\end{cases}
\]
Here, $\mu$ stands for the M\"obius function and the coefficients $\kappa_\alpha$ are nonnegative and satisfy
\[
\sum_{\alpha\geq0}\frac{\kappa_\alpha}{2^\alpha}=2.
\tag{10}
\]

Since $0\leq\chi_d(m)\leq1$, it follows from property (9) that
\[
\xi_d(n)
\leq
\sum_{\ell\mid n}\mu^2(\ell)
=
2^{\omega(n)},
\]
where $\omega(n)$ is the number of distinct prime divisors of $n$. Using the well-known property
\[
2^{\omega(n)}
=
\sum_{\ell\mid n}\mu^2(\ell),
\]
we obtain
\[
\begin{split}
\sum_{h\leq Y}\xi_d(h)
&\leq
\sum_{h\leq Y}2^{\omega(h)}
\\
&=
\sum_{\ell\leq Y}
\mu^2(\ell)
\left\lfloor\frac{Y}{\ell}\right\rfloor
\\
&\lesssim
Y\log(2Y).
\end{split}
\tag{11}
\]

Thanks to inequalities (8)--(11), we obtain
\[
\begin{split}
P_a(X)
&\lesssim
\sum_{d\mid\Omega}
\sum_{\substack{\alpha\geq0\\2^\alpha d^2\leq X}}
\kappa_\alpha d
\frac{X\log(2X)}{2^\alpha d^2}
\\
&\lesssim
X\log(2X)
\biggl(\sum_{d\mid\Omega}\frac1d\biggr)
\biggl(\sum_{\alpha\geq0}
\frac{\kappa_\alpha}{2^\alpha}\biggr).
\end{split}
\]
Thus
\[
P_a(X)
\lesssim
X\log(2X)\sigma_{-1}(\Omega),
\tag{12}
\]
where
\[
\sigma_{-1}(\Omega):=
\sum_{d\mid\Omega}\frac1d.
\]
%

Put $L:=\lfloor \frac{Q}{4}\rfloor, R:=\lfloor 8Q\rfloor.$ 
Thus, we have that
\begin{equation*}
\sum_{\frac{Q}{4}<q\leq 8Q}
\frac{\rho_{a^2}(q)}q
\leq
\frac{4}{Q} \sum_{q\leq 8Q} \rho_{a^2}(q) = \frac{4}{Q} P_a(8Q).
\end{equation*}
Inequality (12) now easily yields
\begin{equation}
\sum_{\frac{Q}{4}<q\leq 8Q}
\frac{\rho_{a^2}(q)}q
\lesssim
\sigma_{-1}(\Omega)
\log(16Q)
\lesssim
\sigma_{-1}(\Omega)\log(2M).
\tag{13}
\end{equation}

It remains to average $\sigma_{-1}(\Omega)$ over $\frac{A}{\sqrt{2}} <a\leq 2\sqrt{2} A$.
Since $\Omega\mid a$, we know that $\sigma_{-1}(\Omega)\leq\sigma_{-1}(a),$ and therefore
\[
\begin{split}
\sum_{\frac{A}{\sqrt{2}} <a\leq 2\sqrt{2} A}\sigma_{-1}(\Omega)
&\leq
\sum_{\frac{A}{\sqrt{2}} <a\leq 2\sqrt{2} A}
\sum_{d\mid a}\frac1d
\\
&=
\sum_{d\leq 2\sqrt{2} A}\frac1d
\#\{\frac{A}{\sqrt{2}} <a\leq 2\sqrt{2} A:d\mid a\}.
\end{split}
\]
For every $d\leq 2\sqrt{2}A$,
\[
\#\{\frac{A}{\sqrt{2}} <a\leq 2\sqrt{2} A:d\mid a\}
\lesssim
\frac{A}{d}.
\]
Hence
\[
\sum_{\frac{A}{\sqrt{2}} <a\leq 2\sqrt{2}A}\sigma_{-1}(\Omega)
\lesssim
A\sum_{d\leq 2\sqrt{2}A}\frac1{d^2}
\lesssim A.
\tag{14}
\]

Combining properties (7), (13), and (14), we obtain
\[
S_H
\lesssim
AH\log(2M).
\tag{15}
\]
Finally, recall that $h=c-a\leq c\lesssim C,$
so that only the intervals satisfying
$$H=2^s\leq h<2^{s+1} \;\;\; \text{ and } \;\;\; 0\leq s\leq \lceil\log_2(2C)\rceil,$$
can occur. Therefore, by property (15),
\[
\begin{split}
E_{U,V}^{\mathrm{off}}(M) \lesssim \sum_{s=0}^{\lceil\log_2(2C)\rceil} S_{2^s}
&\lesssim
A\log(2M)
\sum_{s=0}^{\lceil\log_2(2C)\rceil}2^s
\\
&\lesssim
AC\log(2M)=\frac{UV\log(2M)}{M}.
\end{split}
\]
Together with the diagonal contribution, this proves property (1), and hence also the inequality $\Sigma_M\lesssim\log(2M).$ Finally,
\[
\sum_{1\leq m\leq2N^2}W^2(m)
\lesssim
\sum_{j=0}^{\lceil\log_2(2N^2)\rceil}\log(2^{j+1})
\lesssim
(\log N)^2,
\]
and we are done.

\section*{Acknowledgments}
The author is grateful to Sergey Kislyakov and to Mikhail Vasilyev for a number of helpful discussions. The author acknowledges the use of ChatGPT 5.6 Plus (OpenAI) as a research tool accessed in July and early August 2026 in exploring and checking intermediate arithmetic estimates of the proof of Theorem~\ref{bochkarev} and in locating relevant references. The author has independently checked the resulting arguments and takes full responsibility for the mathematical content. The main ideas of the proofs belong to the author.

\end{document}